\documentclass[11pt,a4paper]{article}
\usepackage{amsmath,amssymb,amscd,amsthm}
\usepackage{graphics}
\input xy
\usepackage{comment}
\usepackage{tikz}
\usetikzlibrary{arrows}
\usepackage{enumerate}
\usepackage[paperwidth=210mm,paperheight=297mm,a4paper,twoside,top=1in, left=1.2in, right=1.2in, bottom=1in]{geometry}
\usepackage[active]{srcltx}
\usepackage{float}
\usepackage[all]{xy}
\usepackage[dvips]{lscape}
\xyoption{web}
\restylefloat{figure}

\newtheorem{lemma}{Lemma}

\newtheorem{prop}[lemma]{Proposition}
\newtheorem{cor}[lemma]{Corollary}

\newtheorem{defi}[lemma]{Definition}
\newtheorem{thm}[lemma]{Theorem}
\newtheorem{ex}[lemma]{Example}
\newtheorem{rmk}[lemma]{Remark}

\newcommand{\ol}{\overline}

\newcommand{\uh}{\mathcal{U}(\mathfrak{h})}
\newcommand{\ug}{\mathcal{U}(\mathfrak{g})}

\newcommand{\poln}{\mathbb{C}[h_{1}, \ldots , h_{n}]}
\newcommand{\hstar}{\mathfrak{h}^{*}}
\newcommand{\olambda}{\overline{\lambda}}

\newcommand{\bb}[1]{\mathbb{#1}}

\newcommand{\ca}[1]{\mathcal{#1}}
\newcommand{\ep}{\epsilon}

\begin{document}

\title{$\uh$-free modules and coherent families}
\author{Jonathan Nilsson}
\date{}
\maketitle

\begin{abstract}
\noindent We investigate the category of $\uh$-free $\mathfrak{g}$-modules.
 Using a functor from this category to the category of coherent families, we show that
 $\uh$-free modules only can exist when $\mathfrak{g}$ is of type $A$ or $C$.
 We then proceed to classify isomorphism classes of $\uh$-free modules of rank $1$ in type $C$,
 which includes an explicit construction of new simple $\mathfrak{sp}(2n)$-modules.
 Finally, we show how translation functors can be used to obtain simple $\uh$-free modules of higher rank.
\end{abstract}

\section{Introduction}
\label{sec1}
In order to understand the structure of a given module category, a classification of simple modules is very helpful.
However, when $\mathfrak{g}$ is a finite-dimensional simple complex Lie algebra, a classification of simple modules seems beyond reach; 
only when $\mathfrak{g} = \mathfrak{sl}_{2}$ a weak version of such a classification exists, see \cite{Bl,Ma}.
However, some classes of simple $\mathfrak{g}$-modules are well understood. For example, simple weight modules
with finite-dimensional weight spaces are completely classified, they fall into two categories: parabolically induced modules
 and cuspidal modules. Parabolically induced modules include simple finite-dimensional modules\cite{Ca,Dix}, and more generally simple highest weight modules\cite{Dix,Hum2,BGG}.
Simple cuspidal modules were classified by Mathieu in 2000, see~\cite{Mathieu}.
Other well studied classes of simple modules include Whittaker modules\cite{Kos}, and Gelfand-Zetlin modules\cite{DFO}.

Another natural class of modules are the ones where the Cartan subalgebra acts freely.
 Specifically, we let $\mathfrak{M}$ be the full subcategory of $\mathfrak{g}$-Mod consisting of
 modules $M$ such that $Res_{\mathfrak{h}}^{\mathfrak{g}} M \simeq _{\uh} \uh$. In other words, $\mathfrak{M}$ consists of the modules which are free of rank $1$ as $\uh$-modules.
In the paper~\cite{Ni}, isomorphism classes of the objects of $\mathfrak{M}$ were classified for $\mathfrak{g} = \mathfrak{sl}_{n}$ which led to several new families of simple $\mathfrak{sl}_{n}$-modules.
 Some of these modules were also studied in connection to the Witt algebra in~\cite{TZ1}, and classifications of $\uh$-free modules over different Witt algebras were obtained in \cite{TZ2}.

In the present paper, we focus on a similar classification of the category $\mathfrak{M}$ in type $C$. We start by explicitly constructing an object $M_{0}$ of $\mathfrak{M}$, and we proceed to show that
every other isomorphism class of $\mathfrak{M}$ can be obtained by twisting $M_{0}$ by an automorphism.
This result is achieved by considering the connection between $\mathfrak{M}$ and the coherent
 families of degree $1$. We construct a functor between these categories, and can then rely on the classification of irreducible semisimple coherent families from~\cite{Mathieu} to classify $\mathfrak{M}$.
This line of argument also directly shows that the category $\mathfrak{M}$ is empty for finite-dimensional simple complex Lie algebras of all types other than $A$ and $C$.
 This then completes the classification of $\mathfrak{M}$ for all such Lie algebras. To summarize, we have the following results about the category $\mathfrak{M}$ for simple complex finite-dimensional Lie algebras:

\begin{itemize}
 \item  The category $\mathfrak{M}$ is empty unless $\mathfrak{g}$ is of type $A$ or type $C$.
 \item When $\mathfrak{g}$ is of type $C$, there exists an object $M_{0} \in \mathfrak{M}$ (definition in Theorem~\ref{mnolldef}) such that any object of $\mathfrak{M}$ is isomorphic to $M_{0}^{\varphi}$ (twist by automorphism) for
 some explicitly given $\varphi \in Aut(\mathfrak{g})$. See Theorem~\ref{mainthm}.
 \item When $\mathfrak{g}$ is of type $A_{n}$, a classification of $\mathfrak{M}$ was obtained in~\cite{Ni}. In the context of this paper, this would be formulated as: there exist an explicitly given family of modules
 $\{M_{b}^{S}\}$  parametrized by $b \in \bb{C}$ and $S \subset \{1, \ldots, n\}$ such that for any object $M$ of $\mathfrak{M}$, there exist $\varphi \in Aut(\mathfrak{g})$ 
 such that $M^{\varphi}$ is isomorphic to some $M_{b}^{S}$. See~\cite{Ni} for details.
\end{itemize}

Here follows a brief summary of the paper. 
 Section~\ref{sec2} deals with the relationship between $\uh$-free modules and coherent families.
 In Section~\ref{sec2dot1} we briefly discuss the category $\mathfrak{M}$ and give an example of one of its objects. Section~\ref{sec3} reminds the reader of the notion of a coherent family, and it lists some
 known results about these. In Section~\ref{weighting} we construct an endofunctor $\ca{W}$ on
 $\mathfrak{g}$-Mod and prove that its image of $\mathfrak{M}$ lies in the set of
 coherent families of rank $1$, which proves the first point above.
 Section~\ref{sec5} deals with the classification of $\mathfrak{M}$ in type $C$. In Section~\ref{sec3dot2} we explicitly construct a simple object $M_{0}$ of
 $\mathfrak{M}$, and in Section~\ref{sec3dot3} we proceed by describing the submodule structure of $\ca{W}(M_{0})$ and of its semisimplification $\ca{W}(M_{0})^{ss}$. 
 Section~\ref{sec3dot4} discusses twisting modules by a family of automorphisms of $\mathfrak{g}$, which eventually
 leads to the proof in Section~\ref{sec3dot5} of the second point above. Finally, in Section~\ref{sec6} we show that by applying translation functors (see~\cite{BG}) to $\mathfrak{M}$ in type $C$, we can obtain
 simple modules which is $\uh$-free of finite rank higher than one. This provides a more general but less explicit construction of a larger category of modules.\\

\noindent
{\bf Acknowledgements}
I am very thankful to Volodymyr Mazorchuk for his many helpful remarks and ideas.
I am also particularly grateful to Professor Olivier Mathieu who, during his visit to Uppsala, suggested the connection between $\uh$-free modules and coherent families.
The results of this paper are based on this connection. The remarks of the referee were also appreciated.

\section{$\uh$-free modules and coherent families}
\label{sec2}
\subsection{Modules where the Cartan acts freely}
\label{sec2dot1}
Let $\mathfrak{g}$ be a finite-dimensional simple complex Lie algebra with a fixed Cartan subalgebra $\mathfrak{h}$. Denote by $\Delta$ the root system and let $Q:=\bb{Z}\Delta$ be the root lattice.
We denote the category of all $\ug$-modules by $\ug$-Mod or sometimes just $\mathfrak{g}$-Mod. Denote by $\mathfrak{M}$ the full subcategory of $\ug$-Mod consisting of modules whose restriction to $\uh$ is free of rank one. When $\mathfrak{g}$ is realized as
 a Lie algebra of matrices, we use the notation $e_{i,j}$ to denote the matrix with a single $1$ in position $(i,j)$ and zeroes everywhere else.
\begin{ex}
\label{ex1}
{\rm 
Let $\mathfrak{g}=\mathfrak{sl}_{2}$ and let $h:=\frac{1}{2}(e_{1,1}-e_{2,2})$. Then $\bb{C}[h]$ becomes an $\mathfrak{sl}_{2}$-module under the action given by
\begin{align*}
 h \cdot f(h) &= hf(h)\\
 e_{1,2} \cdot f(h) &= hf(h-1)\\
 e_{2,1} \cdot f(h) &= -hf(h+1).
\end{align*}
Clearly this module lies in $\mathfrak{M}$.
}
\end{ex}

We shall often refer to modules of $\mathfrak{M}$ as $\uh$-free modules of rank one. 
\begin{lemma}
\label{rootact}
 Let $M \in \mathfrak{M}$ and identify $M$ with $\uh$ as vector spaces (and $\uh$-modules), and let $x_{\alpha} \in \mathfrak{g}_{\alpha}$ be a root vector. We then have
\[x_{\alpha} \cdot f = (x_{\alpha} \cdot 1) \sigma_{\alpha}(f),\]
where $\sigma_{\alpha}$ is the algebra automorphism on $\uh$ satisfying $\sigma_{\alpha}(h) = h-\alpha(h)$ for all $h\in \mathfrak{h}$.
\end{lemma}
\begin{proof}
Let $h_{1}, \ldots, h_{n}$ be a basis for $\mathfrak{h}$ such that $\uh \simeq \bb{C}[h_{1}, \ldots, h_{n}]$. It suffices to prove the lemma for all monomials $f \in  \bb{C}[h_{1}, \ldots, h_{n}]$.
For $f=1$ the lemma is trivially true. Assume the lemma holds for all monomials $f$ of (total) degree $k$.
Then for any $f$ of degree $k$, and for any $h_{i}$, we compute
\begin{align*}
 X_{\alpha} \cdot (h_{i}f) &=   X_{\alpha} \cdot (h_{i} \cdot f) \\
 &=  h_{i} \cdot X_{\alpha} \cdot  f + [X_{\alpha} , h_{i}] \cdot f\\
 &=  h_{i} \cdot X_{\alpha} \cdot  f - \sigma_{\alpha}(h_{i})(X_{\alpha} \cdot f)\\
 &=  (h_{i}- \sigma_{\alpha}(h_{i}))  (X_{\alpha} \cdot  f)\\
 &=  (h_{i}- \sigma_{\alpha}(h_{i}))  (X_{\alpha} \cdot  1)\sigma_{\alpha}(f)\\
 &=   (X_{\alpha} \cdot  1)\sigma_{\alpha}(h_{i}f),
\end{align*}
which shows that the lemma also holds for all monomials $f$ of degree $k+1$. By induction the Lemma holds.
\end{proof}
Thus the action of $\ug$ on $M$ is completely determined by the elements $\{x_{\alpha} \cdot 1 | \alpha \in \Delta\}$.
If $\{\epsilon_{1}, \ldots, \epsilon_{n}\}$ is a fixed ordered basis of $\hstar$, we shall write $\sigma_{i}:=\sigma_{\epsilon_{i}}$.\\ 
\subsection{Coherent families}
\label{sec3}
The theory of coherent families was used by Mathieu in 2000 to classify cuspidal weight modules with finite-dimensional weight spaces, see~\cite{Mathieu}.
We shall restate some of the known properties of coherent families. For details we refer to~\cite{Mathieu}.

\begin{defi}\label{cohdef}
Let $\ug_{0}$ be the commutant of $\mathfrak{h}$ in $\ug$. A {\bf coherent family} of degree $d$ is a weight module $\ca{M}=\bigoplus_{\lambda \in \mathfrak{h}^{*}} \ca{M}_{\lambda}$ such that:
\begin{itemize}
 \item $\dim \ca{M}_{\lambda} = d$ for all $\lambda \in \mathfrak{h}^{*}$.
 \item For any $u \in \ug_{0}$, the function $Tr \; u|_{\ca{M}_{\lambda}}$ is polynomial in $\lambda$.  
\end{itemize}
\end{defi}

For each $\lambda \in \mathfrak{h}^{*}$ we shall write $\overline{\lambda}$ for the unique algebra homomorphism $\uh \rightarrow \mathbb{C}$ extending $\lambda$.
Note that the second condition of Definition~\ref{cohdef} can be reformulated as follows: for any $u\in \ug_{0}$ there exist $f_{u} \in \uh$
such that  $Tr\; u|_{\ca{M}_{\lambda}} = \overline{\lambda}(f_{u})$ for all $\lambda \in \hstar$.

An example of a coherent family for $\mathfrak{sl}_{2}$ is given below.
\begin{ex}
\label{ex2}
{\rm 
 Fix $a \in \bb{C}$. Let $\ca{M}(a)$ be the vector space with basis  $\{ v_{\lambda} | \lambda \in \bb{C}\}$.
Define an action of $\mathfrak{sl}_{2}$ on $\ca{M}(a)$ as follows:
\begin{align*}
 (e_{1,1}-e_{2,2}) \cdot v_{\lambda} &= 2\lambda v_{\lambda}\\
 e_{1,2} \cdot v_{\lambda} &= (a+\lambda) v_{\lambda+1}\\
 e_{2,1} \cdot v_{\lambda} &= (a-\lambda) v_{\lambda-1}.
\end{align*}
This is a coherent family of degree $1$. It is isomorphic to the module described in~\cite[p. 549]{Mathieu}.
}
\end{ex}
If $\ca{M}$ is a coherent family (or more generally any weight module) and $S\subset \hstar$ we define 
\[\ca{M}[S]:=\bigoplus_{\lambda \in S} \ca{M}_{\lambda}.\]
For $\mu \in \hstar$ we shall also write $\ca{M}[\mu]:= \ca{M}[\mu + Q]$. Note that in this notation we have $\ca{M}[\{\mu\}] \neq \ca{M}[\mu]$; the left side is
the single weight space $\ca{M}_{\mu}$, while the right side is a submodule of $\ca{M}$ containing $\ca{M}_{\mu}$ as is seen in the following lemma.
\begin{lemma}
\label{cosetsub}
 Let $\ca{M}$ be a coherent family. For each $\mu \in \hstar$, $\ca{M}[\mu]$ is a submodule of $\ca{M}$ and
\[\ca{M} = \bigoplus_{\mu \in h^{*} / Q} \ca{M}[\mu]. \]
\end{lemma}
\begin{proof}
 Since $x_{\alpha} \ca{M}_{\lambda} \subset \ca{M}_{\lambda + \alpha}$ it is clear that each coset of $Q$ (the root lattice) in $\mathfrak{h}^{*}$ corresponds to a submodule of the coherent family.
\end{proof}

Note that the components $\ca{M}[\mu]$ above are not necessarily simple themselves. Consider for example the submodule $\ca{M}(0)[0]$ from Example~\ref{ex2}; it has length $3$.
However, we have the following useful proposition, see~\cite[p.553--554]{Mathieu}.
\begin{prop}
Let $\ca{M}$ be a coherent family.
\begin{enumerate}[$($i$)$]
 \item  For any $\mu \in \hstar$, the module $\ca{M}[\mu]$ has finite length.
 \item   There exists a unique semisimple coherent family $\ca{M}^{ss}$ such that for each $\mu \in \mathfrak{h}^{*}$, the modules
 $\ca{M}[\mu]$ and $\ca{M}^{ss}[\mu]$ have the same simple subquotients. 
\end{enumerate}
  
\end{prop}
The coherent family $\ca{M}^{ss}$ is called the semisimplification of $\ca{M}$. For coherent families of degree $1$ the construction of $\ca{M}^{ss}$ from $ \ca{M}$ can be realized as follows:
For every $\lambda \in \mathfrak{h}$ where the action of $x_{\alpha}$ on $\ca{M}_{\lambda}$ is zero, modify the action of $x_{-\alpha}$ to be zero on $\ca{M}_{\lambda+\alpha}$.
A coherent family is called irreducible if $\ca{M}_{\lambda}$ is simple as a $\ug_{0}$-module for some $\lambda$. For example, nontrivial direct sums of coherent families are still coherent families
 but they are no longer irreducible. 

Semisimple irreducible coherent families are classified in~\cite{Mathieu}. We recall two results, see~\cite[Lemma 5.3, Remark p.586]{Mathieu}.
\begin{prop}
\label{cohprop} We have:
\begin{enumerate}[$($i$)$]
 \item Coherent families exist for Lie algebras of type $A$ and $C$ only.
 \item For a Lie algebra of type $C_{n}$ ($n \geq 2$), there exists a unique semisimple irreducible coherent family of degree $1$ up to isomorphism.
\end{enumerate}
\end{prop}

\subsection{Weighting functor}
\label{weighting}
The following construction was suggested by Olivier Mathieu as a comment on~\cite{Ni}. 
Denote by $Max(\uh)$ the set of maximal ideals of the algebra $\uh$. For $M$ in $\ug$-Mod, 
consider the $\uh$-module \[\ca{W}(M) := \bigoplus_{\mathfrak{m} \in Max(\uh)} M / \mathfrak{m}M = \bigoplus_{\lambda \in \mathfrak{h}^{*}} M / ker(\overline{\lambda})M.\]

\begin{prop}
 $\ca{W}(M)$ becomes an $\ug$-module by defining the action of root vectors as follows: 
\begin{equation}
\label{actiondef}
 x_{\alpha} \cdot (v + ker(\overline{\lambda})M):= (x_{\alpha} \cdot v) + ker(\overline{\lambda + \alpha})M.
\end{equation}
Moreover, the assignment \[\ca{W}:M \mapsto \bigoplus_{\mathfrak{m} \in Max(\uh)} M / \mathfrak{m}M\]
 is functorial: We have a functor $\ca{W}: \ug\text{-Mod} \rightarrow \ug\text{-Mod}$, which maps a homomorphism $f: M \mapsto N$ to the homomorphism $\ca{W}(f)$ defined by $\ca{W}(f): m+\mathfrak{m}M \mapsto f(m)+\mathfrak{m}N$.
\end{prop}
\begin{proof}
For any root vectors $x_{\alpha}$ and $x_{\beta}$ we have
\begin{align*}
x_{\alpha} \cdot x_{\beta} &\cdot (v + ker(\overline{\lambda})M) - x_{\beta} \cdot x_{\alpha} \cdot (v + ker(\overline{\lambda})M)\\
&=x_{\alpha} \cdot x_{\beta}\cdot v + ker(\overline{\lambda+\beta+\alpha})M - x_{\beta} \cdot x_{\alpha}\cdot v + ker(\overline{\lambda+\alpha+\beta})M\\
&=[x_{\alpha},x_{\beta}]\cdot v + ker(\overline{\lambda+\alpha+\beta})M\\
 &=[x_{\alpha},x_{\beta}]\cdot (v + ker(\overline{\lambda})M).
\end{align*}
Checking the same relation for a root vectors and a Cartan element corresponds to taking $\beta= 0$ above. Thus $\ca{W}(M)$ is a $\mathfrak{g}$-module.
On morphisms, $\ca{W}$ clearly preserves composition and identity, so the functoriality claim follows.
\end{proof}

\begin{lemma}
\label{weightinglemma}
The following holds:
\begin{enumerate}[$($i$)$]
 \item\label{weight1}  For any $M$ in $\ug$-Mod, $\ca{W}(M)$ is a weight module.
 \item\label{weight2}  If $M$ is a weight module, then  $\ca{W}(M) \simeq M$.
 \item\label{weight3} $\ca{W} \circ \ca{W} \simeq \ca{W}$.
 \item\label{weight4} Suppose $M$ admits a central character: there exists a homomorphism 
$\chi_{M}: Z(\mathfrak{g}) \rightarrow \bb{C}$, such that for all $v \in M$ and $z \in Z(\mathfrak{g})$ we have $z \cdot v = \chi_{M}(z)v$.
  Then $\ca{W}(M)$ also admits a central character, and $\chi_{M}=\chi_{\ca{W}(M)}$.
 \item\label{weight5} Suppose $M$ admits a generalized central character: there is a central character $\chi_{M}$ such that for all $v \in M$ there is a $k\in \bb{N}$ such that $(z-\chi_{M})^{k}v=0$.
  Then $\ca{W}(M)$ also admits a generalized central character, and $\chi_{M}=\chi_{\ca{W}(M)}$.
\end{enumerate}
\end{lemma}
\begin{proof}
 On $M / \ker(\overline{\lambda})M$, each element of the form $h-\lambda(h)$ acts as zero, so the direct sum decomposition in the definition of $\ca{W}(M)$ is really the weight space decomposition. Thus \eqref{weight1} holds.
 To prove \eqref{weight2}, let $v \in M_{\mu}$ be a weight vector of $M$. Then for all $h \in \mathfrak{h}$, we have
 $(h-\lambda(h))v=(\mu-\lambda)(h)v$, so if $\mu \neq \lambda$, this is nonzero for some $h$ which means that $v \in ker(\olambda)M$ and $v=0$ in the quotient $M / ker(\ol{\lambda})M$. On the other hand,
 if $\lambda = \mu$, then $ker (\olambda)v=0$ which shows that $M_{\mu} \simeq \mathcal{W}(M)_{\mu}$ as $\uh$-modules. Thus $M \simeq \mathcal{W}(M)$ as $\ug$-modules since the action of root vectors coincide.
Statement \eqref{weight3} follows from \eqref{weight1} and \eqref{weight2}. Finally, suppose $M$ admits central character: there exists a homomorphism
 $\chi: Z(\mathfrak{g}) \rightarrow \bb{C}$ such that $zm=\chi(z)m$ for all $z \in Z(\mathfrak{g})$ and $m \in M$. Then, since $Z(\mathfrak{g}) \subset \ug_{0}$,
 by definition we have $z \cdot (m + \mathfrak{m}M)=zm + \mathfrak{m}M = \chi(z)(m + \mathfrak{m}M)$ for all $\mathfrak{m} \in Max(\uh)$. Thus the central character of $\ca{W}(M)$ is also $\chi$, and \eqref{weight4} holds.
 Similarly, for each $k \in \bb{N}$, we have $(z-\chi(z))^{k} \in \ug_{0}$ so the same argument handles the proof of \eqref{weight5}.
\end{proof}

The classification in the next section relies on the following proposition.
\begin{prop}
\label{freegivescf}
 Let $M$ be a $\uh$-free module of rank $d$, that is,  $Res_{\uh}^{\ug} M \simeq \mathfrak{\uh}^{\oplus d}$. Then $\ca{W}(M)$ is a coherent family of degree $d$. 
\end{prop}
\begin{proof}
By Lemma~\ref{weightinglemma}, $\ca{W}(M)$ is a weight module and $\ca{W}(M)_{\lambda} = M / ker(\overline{\lambda})M$. The (classes of the) $d$ generators of $M$ is a basis in this space, and so
 $\dim \ca{W}(M)_{\lambda} = d$ for all $\lambda \in \mathfrak{h}^{*}$. Now let $u  \in \ug_{0}$. We shall show that there exists $f_{u} \in \uh$ such that $Tr \; u|_{\ca{W}(M)_{\lambda}} = \overline{\lambda}(f_{u})$.
 Since $u$ commutes with $\uh$ we have an endomorphism of $M$ defined by $m \mapsto um$. However, we have usual isomorphisms
\[End_{\mathcal{U}(\mathfrak{h})}(M) \simeq End_{\mathcal{U}(\mathfrak{h})}(\mathcal{U}(\mathfrak{h})^{\oplus d}) \simeq Mat_{n\times n}(\mathcal{U}(\mathfrak{h}))^{op},\] so we fix such isomorphisms and
 we write $[u]=(f_{i,j}^{(u)})$ for the matrix corresponding to the endomorphism given by the multiplication by $u$.
But then the action of $u$ on $M / ker(\overline{\lambda})M$ is given by the matrix $(\overline{\lambda}(f_{i,j}^{(u)}))$ (with respect to the basis of $M / ker(\overline{\lambda})M$ given by the generators of $M$).
Thus $Tr \; u|_{\ca{W}(M)_{\lambda}} = \overline{\lambda}(\sum_{i=1}^{d} f_{i,i}^{(u)})$, which shows that the trace of $u$ on $\ca{W}(M)_{\lambda}$ is polynomial in $\lambda$.
\end{proof}
\begin{cor}
 $\uh$-free modules of finite rank exist only in type $A$ and $C$.
\end{cor}
\begin{proof}
 This follows directly from Proposition~\ref{cohprop} and Proposition~\ref{freegivescf}.
\end{proof}
In the next section we shall go in the opposite direction in type $C$: given a coherent family $\ca{M}$ of degree $1$ we shall obtain all $\uh$-free modules $M$ such that $\ca{W}(M)^{ss} \simeq \ca{M}^{ss}$. 

\section{Classification of $\uh$-free modules of rank $1$ in type $C_{n}$}
\label{sec5}
From here on, we fix $\mathfrak{g}:=\mathfrak{sp}(2n)$; the complex symplectic Lie algebra of rank $n$.
\subsection{A basis of $\mathfrak{sp}(2n)$}
The Lie algebra $\mathfrak{sp}(2n)$ of type $C_{n}$ is the Lie subalgebra of $\mathfrak{gl}_{2n}$ consisting of all $2n\times 2n$-matrices $A$ satisfying $SA=-A^{T}S$ where 
\begin{displaymath}
 S=\left( \begin{array}{cc}
  0 & I_{n}\\
  -I_{n} & 0\\
 \end{array} \right).
\end{displaymath}
Equivalently, $\mathfrak{sp}(2n)$ consists of all $2n\times 2n$-matrices with block form
\begin{displaymath}
 \left( \begin{array}{cc}
  E_{11} & E_{12}\\
  E_{21} & E_{22}\\
 \end{array} \right)
\end{displaymath}
such that $E_{12} = E_{12}^{T}$,  $E_{21} = E_{21}^{T}$ and $E_{22} = -E_{11}^{T}$. We fix the Cartan subalgebra as the subalgebra of $\mathfrak{sp}(2n)$ consisting of diagonal matrices.
We fix the following basis for the Cartan subalgebra: $\{h_{i}:=e_{i,i} - e_{n+i,n+i} | 1 \leq i \leq n\}$.
Let $\{\epsilon_{i}\}$ be the basis of $\mathfrak{h}^{*}$ dual to $\{h_{i}\}$, that is, $\epsilon_{i}(h_{k})=\delta_{i,k}$. The Killing form on $\hstar$ is given by $(\epsilon_{i},\epsilon_{j})=\delta_{i,j}$.
Using this notation, the root system of $\mathfrak{sp}(2n)$ is precisely \[\Delta= \{\pm \epsilon_{i} \pm \epsilon_{j} | 1 \leq i,j \leq n\} \setminus \{0\}.\]

We fix root vectors in $\mathfrak{sp}(2n)$ as follows. The indices $i,j$ are distinct integers between $1$ and $n$.
\begin{displaymath}
\begin{array}{r c l | c}
\multicolumn{3}{c}{\text{Root vector}} &\text{Root} \\
\hline
X_{2\epsilon_{i}}              &:=& 2e_{i,n+i}                 & 2\epsilon_{i}\\
X_{-2\epsilon_{i}}             &:=& -2e_{n+i,i}                 & -2\epsilon_{i}\\
X_{\epsilon_{i}+\epsilon_{j}}  &:=& e_{i,n+j} + e_{j,n+i}      & \epsilon_{i}+\epsilon_{j}\\
X_{-\epsilon_{i}-\epsilon_{j}} &:=& -e_{n+i,j} - e_{n+j,i}      & -\epsilon_{i}-\epsilon_{j}\\
X_{\epsilon_{i}-\epsilon_{j}}  &:=& e_{i,j} - e_{n+j,n+i}      & \epsilon_{i}-\epsilon_{j}\\ 
\end{array}
\end{displaymath}
We have now fixed a basis of $\mathfrak{sp}(2n)$ of the form $B:=\{X_{\alpha} | \alpha \in \Delta\} \cup \{h_{i} | 1 \leq i \leq n\}$.

\subsection{A rank one $\uh$-free module for $\mathfrak{sp}(2n)$}
\label{sec3dot2}
As associative algebras, $\uh \simeq \bb{C}[h_{1}, \ldots, h_{n}]$, so we will define a $\uh$-free module of rank $1$ by extending the canonical action of $\uh$ on $\uh$ to $\ug$.

\begin{thm}
\label{mnolldef}
Let $\sigma_{i}$ be the algebra automorphism of $\bb{C}[h_{1}, \ldots, h_{n}]$ determined by $h_{k} \mapsto h_{k}-\delta_{i,k}$.
The following table provides a $\ug$-module structure on $\bb{C}[h_{1}, \ldots, h_{n}]$.
\begin{displaymath}
\begin{array}{|rcl|}
\hline
h_{i} \cdot f                 &=&   h_{i}f \\
X_{2\epsilon_{i}} \cdot f             &=&   (h_{i}-\frac{1}{2})(h_{i}-\frac{3}{2})\sigma_{i}^{2}(f) \\
X_{-2\epsilon_{i}} \cdot f             &=&   \sigma_{i}^{-2}(f)   \\
X_{\epsilon_{i}+\epsilon_{j}}\cdot f  &=&   (h_{i}-\frac{1}{2})(h_{j}-\frac{1}{2})\sigma_{i}\sigma_{j}(f)\\
X_{-\epsilon_{i}-\epsilon_{j}}\cdot f  &=&   \sigma_{i}^{-1}\sigma_{j}^{-1}(f)      \\
X_{\epsilon_{i}-\epsilon_{j}}\cdot f  &=&   (h_{i}-\frac{1}{2})\sigma_{i}\sigma_{j}^{-1}(f)\\ 
\hline
\end{array}
\end{displaymath}
We denote this $\ug$-module by $M_{0}$.
\end{thm}
\begin{proof}
To show that $M_{0}$ is a $\ug$-module, it suffices to check that
\begin{equation}
 \label{eq1}
    [X, Y] \cdot f = X\cdot Y\cdot f - Y \cdot X \cdot f,
\end{equation}
for all $X,Y \in B$.

Since \eqref{eq1} obviously holds for $X,Y \in \mathfrak{h}$, we let $\alpha \in \Delta$ and $h \in \mathfrak{h}$ and we compute 
\begin{align*}
 h \cdot X_{\alpha} \cdot f -  X_{\alpha} \cdot  h \cdot f &= h(X_{\alpha} \cdot 1) \sigma_{\alpha}(f) - (X_{\alpha} \cdot 1) \sigma_{\alpha}(h)\sigma_{\alpha}(f)\\
 &= (h - \sigma_{\alpha}(h))(X_{\alpha} \cdot 1) \sigma_{\alpha}(f)\\
 &= \alpha(h)(X_{\alpha} \cdot 1) \sigma_{\alpha}(f)\\
 &= \alpha(h)X_{\alpha} \cdot f\\
 &= [h,X_{\alpha}] \cdot f.
\end{align*}

Thus it only remains to check that 
\begin{equation}
 \label{eq2}
    [X_{\alpha}, X_{\beta}] \cdot f = X_{\alpha}\cdot X_{\beta}\cdot f - X_{\beta} \cdot X_{\alpha} \cdot f,
\end{equation}
for all $\alpha,\beta \in \Delta$.
Since $\sigma_{i}$ fixes all $h_{j}$ for $i \neq j$, it is clear that \eqref{eq2} also holds trivially for the following (unordered) pairs of roots $(\alpha,\beta)$:
\begin{align*}
&(2\epsilon_{i},2\epsilon_{j}),(2\epsilon_{i},-2\epsilon_{j}), (2\epsilon_{i},\epsilon_{j}+\epsilon_{k}),(2\epsilon_{i},-\epsilon_{j}-\epsilon_{k}), (2\epsilon_{i},\epsilon_{j}-\epsilon_{k}),\\
&(-2\epsilon_{i},-2\epsilon_{j}),(-2\epsilon_{i},\epsilon_{j}+\epsilon_{k}),(-2\epsilon_{i},-\epsilon_{j}-\epsilon_{k}), (-2\epsilon_{i},\epsilon_{j}-\epsilon_{k}),\\
&(\ep_{i}+\ep_{j},\ep_{k}+\ep_{l}),(\ep_{i}+\ep_{j},-\ep_{k}-\ep_{l}),(\ep_{i}+\ep_{j},\ep_{k}-\ep_{l}),\\
&(-\ep_{i}-\ep_{j},-\ep_{k}-\ep_{l}),(-\ep_{i}-\ep_{j},\ep_{k}-\ep_{l}),(\ep_{i}-\ep_{j},\ep_{k}-\ep_{l}),\\
\end{align*}
where $i,j,k,l$ are pairwise distinct. Note that in these cases, the left side of \eqref{eq2} is zero since $\alpha+\beta$ is not a root.

The remaining cases need a short verification. We show the computation for the pair $(-\ep_{i}-\ep_{j},\ep_{k}-\ep_{l})$ where $i \neq j$ and $k \neq l$. We have
\begin{align*}
 X_{-\ep_{i}-\ep_{j}} &\cdot X_{\ep_{k}-\ep_{l}} \cdot f -  X_{\ep_{k}-\ep_{l}} \cdot X_{-\ep_{i}-\ep_{j}} \cdot f \\
&=  \big( \sigma_{i}^{-1}\sigma_{j}^{-1}(h_{k}-\frac{1}{2}) - (h_{k}-\frac{1}{2})\big) \sigma_{i}^{-1}\sigma_{j}^{-1}\sigma_{k}\sigma_{l}^{-1} f\\
&=  \big((h_{k}-\frac{1}{2} + \delta_{k,i} + \delta_{k,j}) - (h_{k}-\frac{1}{2})\big) \sigma_{i}^{-1}\sigma_{j}^{-1}\sigma_{k}\sigma_{l}^{-1} f\\
&= (\delta_{k,i} + \delta_{k,j}) \sigma_{i}^{-1}\sigma_{j}^{-1}\sigma_{k}\sigma_{l}^{-1} f\\
&= \delta_{k,i}\sigma_{j}^{-1}\sigma_{l}^{-1} f + \delta_{k,j} \sigma_{i}^{-1}\sigma_{l}^{-1} f\\
&= (\delta_{k,i} X_{-\ep_{j}-\ep_{l}} + \delta_{k,j}X_{-\ep_{i}-\ep_{l}} ) \cdot f\\
&=  [X_{-\ep_{i}-\ep_{j}}, X_{\ep_{k}-\ep_{l}}] \cdot f.
\end{align*}

We omit verification of the remaining relations, the computation is similar.
\end{proof}
\begin{prop}
\label{simpleprop}
 The module $M_{0}$ is a simple $\mathfrak{g}$-module.
\end{prop}
\begin{proof}
It is clear that $Res_{\uh}^{\ug}(M_{0}) \simeq \uh$, so $M_{0} \in \mathfrak{M}$. To prove simplicity we note that
the element $(1-X_{-2\ep_{i}})$ of $\mathcal{U}(\mathfrak{sp}(2n))$ acts by decreasing the $h_{i}$-degree of a monomial in $\poln$ by $1$. Thus any nonzero element can be reduced to the generator $1$ of $\poln$, so there
 are no proper nontrivial submodules.
\end{proof}

\subsection{Structure of $\ca{W}(M_{0})$ and  $\ca{W}(M_{0})^{ss}$}
\label{sec3dot3}
For each $\lambda \in \mathfrak{h}^{*}$, let $v_{\lambda}:=1 + ker(\overline{\lambda})M_{0} \in \ca{W}(M_{0})$. Then $\{v_{\lambda} | \lambda \in \hstar\}$ is a basis for $\ca{W}(M_{0})$.

\begin{prop}
The action of $\mathfrak{sp}(2n)$ on weight vectors of the module $\ca{W}(M_{0})$ is given in the table below.
\begin{displaymath}
\begin{array}{|rcl|}
\hline
h_{i} \cdot v_{\lambda}                 &=&   \lambda(h_{i})v_{\lambda}\\
X_{2\epsilon_{i}}  \cdot v_{\lambda}             &=&   (\lambda(h_{i})+\frac{3}{2})(\lambda(h_{i})+\frac{1}{2})v_{\lambda + 2\epsilon_{i}} \\
X_{-2\epsilon_{i}}  \cdot v_{\lambda}             &=&   v_{\lambda - 2\epsilon_{i}}   \\
X_{\epsilon_{i}+\epsilon_{j}}\cdot v_{\lambda}  &=&   (\lambda(h_{i})+\frac{1}{2})(\lambda(h_{j})+\frac{1}{2})v_{\lambda + \epsilon_{i} + \epsilon_{j}}\\
X_{-\epsilon_{i}-\epsilon_{j}}\cdot v_{\lambda}  &=&   v_{\lambda - \epsilon_{i} - \epsilon_{j}}\\
X_{\epsilon_{i}-\epsilon_{j}}\cdot v_{\lambda}  &=&   (\lambda(h_{i})+\frac{1}{2})v_{\lambda + \epsilon_{i} - \epsilon_{j}}\\
\hline
\end{array}
\end{displaymath}
\end{prop}
\begin{proof}
This follows from Theorem~\ref{mnolldef} together with \eqref{actiondef}. 
\end{proof}

By Lemma~\ref{cosetsub}, $\ca{W}(M_{0})$ is the direct sum of its submodules $\ca{W}(M_{0})[\mu]$. We shall now describe the structure of each such component.
For each nonzero $\alpha \in \hstar$ we have a corresponding plane \[P_{\alpha}:=\{ \lambda \in \hstar | (\lambda,\alpha) =0\}\] orthogonal to $\alpha$. We also have the corresponding real half space
 \[S_{\alpha}:=\{ \lambda \in \hstar | (\lambda,\alpha)  \in \bb{R}^{\geq 0}\}.\]

\begin{prop}
\label{propsubs}
Let $\mu \in P_{\epsilon_{i}} + (\bb{Z} +\frac{1}{2})\epsilon_{i}$. Then 
 $\ca{W}(M_{0})[(\mu+Q) \cap S_{-\epsilon_{i}}]$ is a proper nontrivial submodule of $\ca{W}(M_{0})[\mu]$, and there is an exact sequence
\[0 \longrightarrow \ca{W}(M_{0})[(\mu+Q) \cap S_{-\epsilon_{i}}]  \longrightarrow \ca{W}(M_{0})[\mu]  \longrightarrow \ca{W}(M_{0})[(\mu+Q) \cap S_{\epsilon_{i}}] \longrightarrow 0.\]
\end{prop}
\begin{proof}
Let $\mu \in P_{\epsilon_{i}} + (\bb{Z} +\frac{1}{2})\epsilon_{i}$. Note that $\mu+Q = Supp(\ca{W}(M_{0})[\mu])  \subset S_{-\ep_{i}} \cup S_{\ep_{i}}$, and that $S_{-\ep_{i}} \cap S_{\ep_{i}}=P_{\ep_{i}}$.
 Suppose there exists $v_{\lambda} \in \ca{W}(M_{0})[(\mu+Q) \cap S_{-\epsilon_{i}}]$ and a root vector $x_{\alpha}$ such that $\alpha + \lambda \in (\mu+Q) \cap S_{\epsilon_{i}}$.
 Then there are two possibilities, either $\lambda(h_{i})=-\frac{1}{2}$ and
 $\alpha=\epsilon_{i} \pm \epsilon_{j}$, or $\lambda(h_{i})=-\frac{3}{2}$ and $\alpha=2\epsilon_{i}$. In either case, $x_{\alpha}v_{\lambda}=0$ by the above table.
 Thus $\ca{W}(M_{0})[(\mu+Q) \cap S_{-\epsilon_{i}}]$ is a submodule. Since $(\mu+Q) \cap P_{\ep_{i}} = \varnothing$, the quotient is clearly isomorphic to $\ca{W}(M_{0})[(\mu+Q) \cap S_{\epsilon_{i}}]$.
\end{proof}
 In $\ca{W}(M_{0})^{ss}[\mu]$ the corresponding short exact sequence is split, so we have the following corollaries.
\begin{cor}
\label{corsub}
For $\mu \in P_{\epsilon_{i}} + (\bb{Z} +\frac{1}{2})\epsilon_{i}$ we have
\[\ca{W}(M_{0})^{ss}[\mu] = \ca{W}(M_{0})^{ss}[(\mu+Q) \cap S_{-\epsilon_{i}}] \oplus \ca{W}(M_{0})^{ss}[(\mu+Q) \cap S_{\epsilon_{i}}].\]
Moreover, since $\ca{W}(M_{0})^{ss}$ is the unique semisimple coherent family of degree $1$ by Proposition~\ref{cohprop},
we also have 
\[\ca{W}(M)^{ss}[\mu] = \ca{W}(M)^{ss}[(\mu+Q) \cap S_{-\epsilon_{i}}] \oplus \ca{W}(M)^{ss}[(\mu+Q) \cap S_{\epsilon_{i}}],\]
for any $M \in \mathfrak{M}$. 
\end{cor}

\begin{cor}
\label{corsubs}
Let $\lambda_{0}:=-\frac{1}{2}\sum_{i=1}^{n} \epsilon_{i}$. The module $\ca{W}(M_{0})^{ss}[\lambda_{0}]$ is the direct sum of the $2^{n}$ simple
 submodules $\ca{W}(M_{0})^{ss}[(\lambda_{0}+Q) \bigcap_{i=1}^{n} S_{\pm \epsilon_{i}}]$. As in Corollary~\ref{corsub}, the same holds with $M_{0}$ replaced by any $M \in \mathfrak{M}$.
\end{cor}

\begin{rmk}
 Note that the modules $\ca{W}(M_{0})[\lambda]$ discussed above are all indecomposable projective modules in the category
of weight $\mathfrak{sp}(2n)$-modules with bounded weight multiplicities. This class of modules was studied in \cite{GS}.
\end{rmk}

\begin{ex}
\label{expic}
{\rm 
Here follows an attempt to visualize the situation for $n=2$. Consider the module $\ca{W}(M_{0})$.
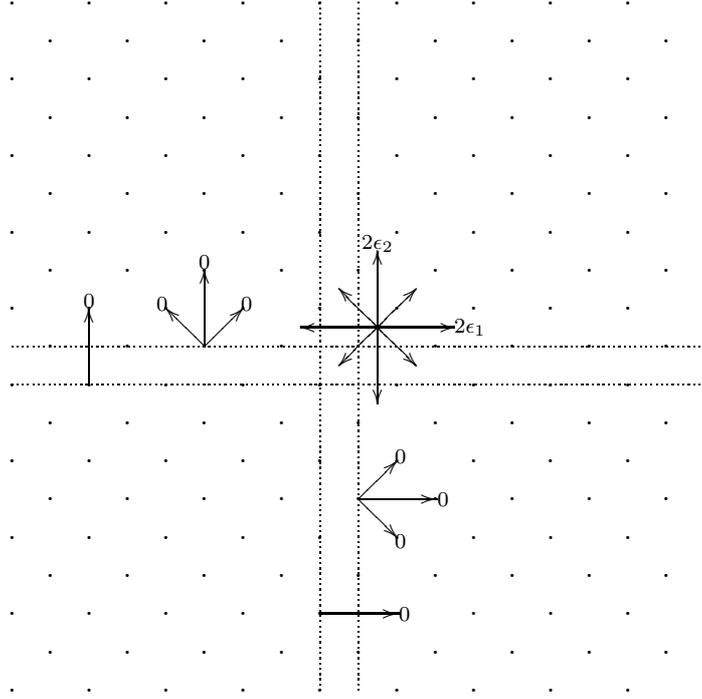
\[
\def\latticebody{%
\ifnum\latticeA=0 \ifnum\latticeB=0\drop{.} %
\else\drop{.}\fi\else\drop{.}\fi}
\xy *\xybox{0;<1.2pc,1.2pc>:<-1.2pc,1.2pc>::
,0,{\croplattice{-9}9{-9}9{-9}9{-9}9}
 ,(0.5,0),{\ar@{}|(1.2){\times} (0.5,0)}
 ,(-4.5,-4.5),{\ar@{.} (4.5,4.5)}
 ,(-5,-4),{\ar@{.} (4,5)}
 ,(-4.5,4.5),{\ar@{.} (4.5,-4.5)}
 ,(-5,4),{\ar@{.} (4,-5)}
 ,(-2,2),{\ar@{->}|(1.1){0} (-1,3)}
 ,(-2,2),{\ar@{->}|(1.1){0} (-1,2)}
 ,(-2,2),{\ar@{->}|(1.1){0} (-2,3)}
 ,(-4,3),{\ar@{->}|(1.1){0} (-3,4)}
 ,(-2,-2),{\ar@{->}|(1.1){0} (-1,-3)}
 ,(-2,-2),{\ar@{->}|(1.1){0} (-1,-2)}
 ,(-2,-2),{\ar@{->}|(1.1){0} (-2,-3)}
 ,(-4,-3),{\ar@{->}|(1.1){0} (-3,-4)}
 ,(0.5,0),{\ar@{->} (1.5,0)}
 ,(0.5,0),{\ar@{->} (0.5,1)}
 ,(0.5,0),{\ar@{->} (-0.5,0)}
 ,(0.5,0),{\ar@{->} (0.5,-1)}
 ,(0.5,0),{\ar@{->}|(1.1){2\ep_{2}} (1.5,1)}
 ,(0.5,0),{\ar@{->} (-0.5,-1)}
 ,(0.5,0),{\ar@{->}|(1.2){2\ep_{1}} (1.5,-1)}
 ,(0.5,0),{\ar@{->} (-0.5,1)}}
\endxy
\]
The picture above describes the submodule $\ca{W}(M_{0})[\lambda_{0}]$. The picture is of $\hstar$, or more precisely, its real affine subspace $\lambda_{0}+\mathbb{R}\Delta$.
 The root system of $\mathfrak{sp}(4)$ is pictured in the center. The dots indicate the support of  $\ca{W}(M_{0})[\lambda_{0}]$, so that each dot corresponds to a one-dimensional weight space.
 The dotted lines are the four hyperplanes 
\[\{ \lambda \in \hstar | \lambda(h_{1})=-\frac{1}{2}\}, \quad \{ \lambda \in \hstar | \lambda(h_{1})=-\frac{3}{2}\}\]
\[\{ \lambda \in \hstar | \lambda(h_{2})=-\frac{1}{2}\}, \text{ and } \{ \lambda \in \hstar | \lambda(h_{2})=-\frac{3}{2}\}.\]
The arrows with zeroes going from the hyperplanes indicate that the action of the corresponding root vector is zero on that hyperplane
 (compare with the table above giving the action of root vectors on  $\ca{W}(M_{0})$). Outside the four hyperplanes, the action of all root vectors on weight spaces are bijective.
We see from the picture that the basis vectors corresponding to dots contained in the left half space of the picture span a submodule of $\ca{W}(M_{0})[\lambda_{0}]$.
 This submodule is precisely $\ca{W}(M_{0})[(\lambda_{0}+Q) \cap S_{-\ep_{1}}]$ from Proposition~\ref{propsubs}. Similarly, the bottom half space of the picture corresponds to the submodule
  $\ca{W}(M_{0})[(\lambda_{0}+Q) \cap S_{-\ep_{2}}]$. The minimal nonzero submodule of $\ca{W}(M_{0})[\lambda_{0}]$ is 
 $\ca{W}(M_{0})[(\lambda_{0}+Q) \cap S_{-\ep_{1}} \cap S_{-\ep_{2}}]$ which corresponds to the lower left quadrant. It is the intersection of the two aforementioned submodules.
 From the picture it is also clear that $\ca{W}(M_{0})[\lambda_{0}]$ has length four with subquotients corresponding to support contained in the four quadrants.
The situation above is comparable to what was studied in the paper~\cite{BKLM}.
}
\end{ex}

\subsection{Twisting $M_{0}$ by automorphisms}
\label{sec3dot4}
In general, if $(M, \cdot)$ is a module and $\varphi \in Aut(\mathfrak{g})$ we can define a new action $\bullet$ of $\mathfrak{g}$ on $M$ by $x \bullet m := \varphi(x) \cdot m$.
The resulting module is denoted $M^{\varphi}$. Moreover, if $M$ is $\uh$-free and $\varphi(\mathfrak{h})=\mathfrak{h}$ we obtain a $\uh$-free module $M^{\varphi}$.

The following proposition describes two important types of automorphisms.
 We shall write $min(M)$ for the unique minimal nonzero submodule of $M$ (when it exists).
\begin{prop}
\label{twistprop}
Let $M \in \mathfrak{M}$.
\begin{enumerate}
 \item\label{eq7} Let $\varphi=exp(ad \: h_{0})$ for some fixed $h_{0} \in \hstar$. Then for $\alpha \in \Delta$ we have 
$\varphi(X_{\alpha})= e^{\alpha(h_{0})} X_{\alpha}$ and $\varphi$ fixes $\mathfrak{h}$ pointwise. In particular, this implies that 
\[Supp(min(\ca{W}(M^{\varphi})[\lambda_{0}]))=Supp(min(\ca{W}(M)[\lambda_{0}])).\]

 \item\label{eq8} For each $1 \leq k \leq n$, let $\varphi_{k}:= exp(ad \: X_{k})exp(-ad \: X_{-k})exp(ad \: X_{k})$,
 where $X_{k} := \frac{1}{2} X_{2\ep_{k}}$ and $X_{-k} := -\frac{1}{2} X_{-2\ep_{k}}$. 
Then $\varphi_{k}$ stabilizes $\mathfrak{h}$ and for $\alpha \in \Delta$ we have $\varphi_{k}(X_{\alpha}) = \pm X_{w_{k}(\alpha)}$ where $w_{k}$ is the element of the Weyl group corresponding to the simple reflection
 in the hyperplane orthogonal to the long root $2\epsilon_{k}$.  In particular, this implies that 
\[Supp(min(\ca{W}(M^{\varphi_{k}})[\lambda_{0}]))=w_{k} \big( Supp(min(\ca{W}(M)[\lambda_{0}])) \big).\]
\end{enumerate}
\end{prop}
\begin{proof}
 Since $\mathfrak{h}$ is commutative, all but the first term of $exp(ad \: h_{0})$ is zero on $\mathfrak{h}$, so on $\mathfrak{h}$, $\varphi$ is the identity. Next we have
\[exp(ad \: h_{0})(X_{\alpha}) = \sum_{k=0}^{\infty} (ad \: h_{0})^{k}(X_{\alpha}) = \sum_{k=0}^{\infty} \alpha(h_{0})^{k} X_{\alpha} = e^{\alpha(h_{0})} X_{\alpha},\]
as stated. Thus the action of $X_{\alpha}$ on $\ca{W}(M^{\varphi})$ is just a rescaling of its action on $\ca{W}(M)$. Thus if we identify $\ca{W}(M^{\varphi})$ and $\ca{W}(M)$ as sets,
 they will have the same submodules. Thus \eqref{eq7} holds.\\
 
To prove \eqref{eq8}, we first note that $[X_{k},X_{-k}] = h_{k}$, $[h_{k},X_{k}] = 2X_{k}$, $[h_{k},X_{-k}] = -2X_{-k}$. Thus it is easy to compute:
\begin{align*}
exp(ad \: X_{k})&exp(-ad \: X_{k})exp(ad \: X_{k})(h_{k}) =exp(ad \: X_{k})exp(-ad \: X_{-k})(h_{k}-2X_{k})\\
 &= exp(ad \: X_{k})(h_{k}-2X_{k} + [h_{k}-2X_{k},X_{-k}] - [[X_{k},X_{-k}],X_{-k}])\\
 &= exp(ad \: X_{k})(h_{k}-2X_{k} + -2X_{-k}-2h_{k} +2 X_{-k})\\
 &= exp(ad \: X_{k})(-h_{k}-2X_{k})\\
 &= -h_{k}-2X_{k} + [X_{k},-h_{k}-2X_{k}] = -h_{k},
\end{align*}
while clearly $\varphi_{k}(h_{i}) = h_{i}$ for $i \neq k$. Thus $\varphi_{k}$ stabilizes $\mathfrak{h}$.
Similar calculations show that $\varphi_{k}(X_{2\ep_{k}}) = X_{-2\ep_{k}}$ and $\varphi_{k}(X_{-2\ep_{k}}) = X_{2\ep_{k}}$.
One can check that for $i \neq k$ we have $\varphi_{k}(X_{\ep_{k} \pm \ep_{i}}) = X_{-\ep_{k} \pm \ep_{i}} $
 and $\varphi_{k}(X_{-\ep_{k} \pm \ep_{i}}) = X_{\ep_{k} \pm \ep_{i}}$.
We clearly also have $\varphi_{k}(X_{\alpha}) = X_{\alpha}$ whenever $\alpha$ is orthogonal to $\ep_{k}$.
It follows that $\varphi_{k}$ precisely acts on root vectors by Weyl group element $w_{k}$ corresponding to reflection in $P_{\ep_{k}}$: we have $\varphi_{k}(X_{\alpha}) = X_{w_{k}\alpha}$ for all $\alpha \in \Delta$.
Thus it follows that $\ca{W}(M)[(\lambda_{0}+Q) \cap S_{\ep_{i}}]$ is a submodule of $\ca{W}(M)[\lambda_{0}]$ if and only if
 $\ca{W}(M^{\varphi_{k}})[(\lambda_{0}+Q) \cap w_{k}S_{\ep_{i}}]$ is a submodule of $\ca{W}(M^{\varphi_{k}})[\lambda_{0}]$. Claim \eqref{eq8} follows.
\end{proof}

\subsection{Uniqueness of $M_{0}$ up to twisting}
\label{sec3dot5}
Define a relation $\sim$ on $\mathfrak{M}$ by $M \sim M'$ if and only if there exists $\varphi \in Aut(\mathfrak{sp}(2n))$ such that $M' \simeq M^{\varphi}$.
 This is an equivalence relation. 

\begin{lemma}
\label{scaling}
 Let $M, M' \in \mathfrak{M}$ be two modules, both identified with $\poln$ as $\uh$-modules. Suppose that for each root $\alpha \in \Delta$ there exist a nonzero constant $c_{\alpha}$ such that
 $x_{\alpha} \cdot 1_{M'} =c_{\alpha} ( x_{\alpha} \cdot 1_{M})$. Then $M \sim M'$.
\end{lemma}
\begin{proof}
 Let $\Sigma$ be a basis for $\Delta$. Then all constants $c_{\alpha}$ are determined by $\{c_{\alpha} | \alpha \in \Sigma\}$. Since $c_{\alpha} \neq 0$ and since we have $\dim \mathfrak{h} =n= |\Sigma|$,
 there exists $h \in \mathfrak{h}$ such that $c_{\alpha} = e^{\alpha(h)}$ for all $\alpha \in \Sigma$. But then $M \sim M'$ since $M' \simeq M^{\varphi}$ with $\varphi = exp(ad \; h)$ (compare with Proposition~\ref{twistprop}).
\end{proof}

\begin{thm}
\label{mainthm}
In type $C$, any $\uh$-free module of rank $1$ is isomorphic to $M_{0}^{\varphi}$ for some $\varphi \in Aut(\mathfrak{sp}(2n))$.
\end{thm}
\begin{proof}
Let $M \in \mathfrak{M}$. Let $N$ be the unique minimal nonzero submodule of $\ca{W}(M)[\lambda_{0}]$ (see Corollary~\ref{corsubs}). Define $\Omega:=\{1 \leq i \leq n | Supp(N) \subset S_{\epsilon_{i}}\}$.
Note that the automorphisms $\varphi_{k}$ from Proposition~\ref{twistprop} stabilize $\mathfrak{h}$ and they commute, which lets us define $\varphi_{\Omega} := \prod_{i \in \Omega} \varphi_{i}$. 
We then have $M':=M^{\varphi_{\Omega}} \in \mathfrak{M}$, and $N':=\ca{W}(M')[(\lambda_{0} + Q) \bigcap_{1 \leq i \leq n}S_{-\epsilon_{i}}]$ is a submodule of $\ca{W}(M')[\lambda_{0}]$.
Now, by Lemma~\ref{rootact} we have $X_{2\epsilon_{i}}X_{-2\epsilon_{i}} \cdot 1_{M'}=(X_{2\epsilon_{i}}\cdot 1_{M'})\sigma_{i}^{2}(X_{-2\epsilon_{i}}\cdot 1_{M'})$ for each $1 \leq i \leq n$.
Thus $X_{2\epsilon_{i}}X_{-2\epsilon_{i}}$ acts on $\ca{W}(M')$ and also on $\ca{W}(M')^{ss}$ by 
\[X_{2\epsilon_{i}}X_{-2\epsilon_{i}} \cdot v_{\lambda} = \olambda \big( (X_{2\epsilon_{i}}\cdot 1_{M'})\sigma_{i}^{2}(X_{-2\epsilon_{i}}\cdot 1_{M'}) \big)v_{\lambda},\]
for each $\lambda \in \hstar$. But by Proposition~\ref{cohprop} we have $\ca{W}(M')^{ss} \simeq \ca{W}(M_{0})^{ss}$, and $X_{2\epsilon_{i}}X_{-2\epsilon_{i}}$ acts on $\ca{W}(M_{0})^{ss}$ by
 \[X_{2\epsilon_{i}}X_{-2\epsilon_{i}} \cdot v_{\lambda} = \olambda \big( (h_{i}-\frac{1}{2})(h_{i}-\frac{3}{2}) \big) v_{\lambda}.\]
We conclude that for each $1\leq i \leq n$ we have
 \[(X_{2\epsilon_{i}}\cdot 1_{M'})\sigma_{i}^{2}(X_{-2\epsilon_{i}}\cdot 1_{M'}) = (h_{i}-\frac{1}{2})(h_{i}-\frac{3}{2}).\]
Analogous arguments imply
 \[(X_{\epsilon_{i}+\epsilon_{j}}\cdot 1_{M'})\sigma_{i}\sigma_{j}(X_{-\epsilon_{i}-\epsilon_{j}}\cdot 1_{M'}) = (h_{i}-\frac{1}{2})(h_{j}-\frac{1}{2}),\]
and 
 \[(X_{\epsilon_{i}-\epsilon_{j}}\cdot 1_{M'})\sigma_{i}\sigma_{j}^{-1}(X_{-\epsilon_{i}+\epsilon_{j}}\cdot 1_{M'}) = (h_{i}-\frac{1}{2}).\]
But since $N'$ is a submodule of $\ca{W}(M')$ and $\lambda_{0}\in Supp(N')$, the action of $\mathfrak{sp}(2n)$ on $\ca{W}(M')$ satisfies
 $X_{2\epsilon_{i}} \cdot v_{\lambda_{0}} = X_{\epsilon_{i}+\epsilon_{j}} \cdot v_{\lambda_{0}}=X_{\epsilon_{i}-\epsilon_{j}} \cdot v_{\lambda_{0}}=0$
for all $1 \leq i,j \leq n$, $i \neq j$. Similarly, taking $j\neq i$ we have $\lambda_{0}-\epsilon_{i}-\epsilon_{j}\in Supp(N')$ implying $X_{2\epsilon_{i}} \cdot v_{\lambda_{0}-\epsilon_{i}-\epsilon_{j}}=0$ for all $1 \leq i \leq n$.
Taken together, this shows that $(h_{i}-\frac{1}{2})$ is a factor of $(X_{2\epsilon_{i}}\cdot 1_{M'}), (X_{\epsilon_{i}+\epsilon_{j}}\cdot 1_{M'})$ and $(X_{\epsilon_{i}-\epsilon_{j}}\cdot 1_{M'})$, while
  $(h_{i}-\frac{3}{2})$ is a factor of $(X_{2\epsilon_{i}}\cdot 1_{M'})$. Thus the action of $\mathfrak{sp}(2n)$ on $M'$ is determined up to scalar multiples. Since action of root vectors on $M'$ and $M_{0}$
 differ only by scalar multiples, Lemma~\ref{scaling} implies $M' \sim M_{0}$. By definition we have $M \sim M'$, so we also have $M \sim M_{0}$.
\end{proof}

\begin{cor}
\label{maincor1}
 In type $C$, all objects of $\mathfrak{M}$ are simple and have the same central character.
\end{cor}
\begin{proof}
 This follows directly from Proposition~\ref{simpleprop} and Theorem~\ref{mainthm} since each automorphism (from Theorem~\ref{mainthm}) twisting $M_{0}$ defines an auto-equivalence on $\mathfrak{M}$.
\end{proof}

 Following the notation of~\cite{BM} we recall that a Whittaker pair consists of a pair of two Lie algebras $(\mathfrak{g},\mathfrak{n})$ such that $\mathfrak{n}$ is a subalgebra of $\mathfrak{g}$, 
$\mathfrak{n}$ is quasi-nilpotent,
 and the adjoint action of
 $\mathfrak{n}$ on $\mathfrak{g} / \mathfrak{n}$ is locally nilpotent. A generalized Whittaker module for a fixed Whittaker pair $(\mathfrak{g},\mathfrak{n})$
 is a $\mathfrak{g}$-module $M$ on which the action of $\mathfrak{n}$ is locally finite, in other words $\dim \: \mathcal{U}(\mathfrak{n})v < \infty$ for all $v \in M$. For details, see~\cite{BM}.

\begin{cor}
\label{maincor2}
 In type $C$, all modules of $\mathfrak{M}$ are generalized Whittaker modules. The module $M_{0}^{\varphi}$ is a generalized Whittaker module for the Whittaker pair $(\mathfrak{g},\varphi(\mathfrak{n}))$, where
 $\mathfrak{n}:=Span\{X_{-\epsilon_{i}-\epsilon_{j}} | 1 \leq i,j\leq n\}$.
\end{cor}
\begin{proof}
Note that $\mathfrak{n}$ is a commutative (hence nilpotent) subalgebra of $\mathfrak{sp}(2n)$ of dimension $\frac{1}{2}n(n+1)$.
Moreover, the adjoint action of $\mathfrak{n}$ on $\mathfrak{sp}(2n) / \mathfrak{n}$ is nilpotent. Thus  $(\mathfrak{sp}(2n),\mathfrak{n})$ is a Whittaker pair.
The action of $\mathfrak{n}$ on $M_{0}$ is clearly locally finite by Theorem~\ref{actiondef}, since $\mathfrak{n}$ never increases the degree of a polynomial. This means that
 $M_{0}$ is a generalized Whittaker module for the Whittaker pair $(\mathfrak{g},\mathfrak{n})$.
 Similarly, $M_{0}^{\varphi}$ is a generalized Whittaker module for the Whittaker pair $(\mathfrak{g},\varphi(\mathfrak{n}))$.
\end{proof}

\begin{rmk}
 Note that none of Corollaries~\ref{maincor1} and~\ref{maincor2} hold in type $A$, see~\cite{Ni}.
\end{rmk}

\subsection{Construction of $\uh$-free modules of higher rank}
\label{sec6}
The idea of this section is to apply translation functors to $\uh$-free modules of rank $1$ to obtain $\uh$-free modules of higher rank.
 This is analogous to what Mathieu does for coherent families, see ~\cite[p.584]{Mathieu}, and it produces a large set of $\uh$-free modules.

Let $\Theta = Hom_{\mathfrak{g}}(Z(\mathfrak{g}),\bb{C})$ be the set of central characters. We have a map $\chi: \hstar \rightarrow \Theta$ which maps $\lambda \in \hstar$ to the central character of
the Verma module $M(\lambda)$. For any module $M$, we write $M^{\theta}$ for the maximal submodule of $M$ having generalized central character $\theta$.
Now, for $\theta \in \Theta$ we denote by $\mathfrak{F}(\theta)$ the full subcategory of $\mathfrak{sp}(2n)$-Mod consisting of simple
 modules which are free of finite rank when restricted to $\uh$, and whose central character is $\theta$.
Denote by $\ca{CF}$ the full subcategory of $\mathfrak{sp}(2n)$-Mod consisting of coherent families.
Let $\theta_{0}$ be the central character of $\ca{W}(M_{0})$.
We assume that $\theta$ is such that $\mathfrak{F}(\theta)$ is nonempty. This means that $\ca{W}(\mathfrak{F}(\theta))$ and $\ca{W}(\mathfrak{F}(\theta_{0}))$ both contains coherent families with
 central characters $\theta$ and $\theta_{0}$ respectively, as $\ca{W}$ preserves central characters (Lemma~\ref{weightinglemma}\eqref{weight4}). 

Now in our case, as in~\cite{BG}, we can fix $\Lambda \in \chi^{-1}(\theta_{0})-\chi^{-1}(\theta) \subset \hstar$ as integral and dominant,
 so that the simple highest weight module $L(\Lambda)$ of highest weight $\Lambda$ is finite-dimensional.
It follows as in~\cite[p.584]{Mathieu} that the translation functor 
 $F_{\Lambda}: \mathfrak{sp}(2n)\text{-Mod}(\theta) \rightarrow \mathfrak{sp}(2n)\text{-Mod}(\theta_{0})$ which maps $M \mapsto (M \otimes L(\Lambda))^{\theta_{0}}$ is an equivalence of categories.

\begin{lemma}
 Consider the following diagram in the category of categories:
\[ \xymatrix@=2cm{
 \mathfrak{F}(\theta)        \ar[r]^{F_{\Lambda}}  \ar[d]_{\ca{W}}    & \mathfrak{F}(\theta_{0}) \ar[d]_{\ca{W}}  \\
 \ca{CF} \ar[r]^{F_{\Lambda}}                     & \ca{CF}  } \]
Both horizontal arrows are equivalences of categories, and the diagram commutes: 
for each $M \in \mathfrak{F}(\theta)$ we have $(\ca{W} \circ F_{\Lambda})(M) \simeq (F_{\Lambda} \circ \ca{W})(M)$.
\end{lemma}
\begin{proof}
The $F_{\Lambda}$'s above are restrictions of an equivalence functor (see~\cite{BG}), so to prove that they are equivalences, it suffices to check that $F_{\Lambda}$ maps $\uh$-free modules to $\uh$-free modules,
 and that it maps coherent families to coherent families.
 If $M$ is $\uh$-free of finite rank, then so is clearly $M \otimes L(\Lambda)$, so $(M \otimes L(\Lambda))^{\theta_{0}}$ is projective in $\uh$-Mod. 
 But by the Quillen-Suslin theorem (see for example~\cite{Qu}), every projective module over a polynomial ring is free, so $F_{\Lambda}(M)=(M \otimes L(\Lambda))^{\theta_{0}}$ is $\uh$-free. 
That coherent families are stable under translation functors follows from~\cite{Mathieu}. Thus the horizontal arrows are equivalences.

Now whenever $E$ is a finite dimensional module and $M$ is free over $\uh$ we claim that the $\mathfrak{g}$-modules $\ca{W}(M)\otimes_{\bb{C}} E$ and $\ca{W}(M\otimes_{\bb{C}} E)$ are isomorphic.
We verify this by exhibiting explicit inverse morphisms. Let $\varphi: \ca{W}(M)\otimes_{\bb{C}} E \rightarrow \ca{W}(M\otimes_{\bb{C}} E)$ be defined by
\[\varphi((m+ker(\ol{\mu})M)\otimes v_{\lambda}) = m\otimes v_{\lambda} + ker(\ol{\mu + \lambda})(M\otimes E),\] where $v_{\lambda}$ is a weight vector in $E$ of weight $\lambda$.
Similarly we define $\psi: \ca{W}(M\otimes_{\bb{C}} E) \rightarrow \ca{W}(M)\otimes_{\bb{C}} E$  by
\[\psi(m\otimes v_{\lambda} + ker(\ol{\mu})(M\otimes E)) = (m +ker(\ol{\mu-\lambda})M)\otimes v_{\lambda}.\]
One can now check that $\varphi$ and $\psi$ are mutually inverse $\mathfrak{g}$-module homomorphisms.
Thus $\ca{W}(M)\otimes E \simeq \ca{W}(M\otimes E)$, and in particular if we take $E:=L(\Lambda)$ and project to the block of character $\theta_{0}$
 we have $(\ca{W}(M)\otimes L(\Lambda))^{\theta_{0}} \simeq (\ca{W}(M\otimes L(\Lambda)))^{\theta_{0}}$.
Since $\ca{W}$ preserves central character (Lemma~\ref{weightinglemma}) we also have $(\ca{W}(M)\otimes L(\Lambda))^{\theta_{0}} \simeq \ca{W}((M\otimes L(\Lambda))^{\theta_{0}})$,
 which is to say $(\ca{W} \circ F_{\Lambda})(M) \simeq (F_{\Lambda} \circ \ca{W})(M)$. This completes the proof.
\end{proof}

\begin{cor}
There exist simple $\uh$-free modules of rank higher than $1$.
\end{cor}
\begin{proof}
For a fixed character $\theta \neq \theta_{0}$ as in the lemma, consider the modules of form $F_{\Lambda}^{-1}(M_{0}^{\varphi})$ where $M_{0}^{\varphi}$ is any rank one $\uh$-free module as in Theorem~\ref{mainthm}.
The modules of form $F_{\Lambda}^{-1}(M_{0}^{\varphi})$ are simple objects of $\mathfrak{F}(\theta)$, and since all rank one $\uh$-free modules have central character $\theta_{0}$, the modules $F_{\Lambda}^{-1}(M_{0}^{\varphi})$ 
must be of rank higher than $1$.
\end{proof}

\noindent Department of Mathematics, Uppsala University, Box 480, SE-751 06, Uppsala, Sweden, email: jonathan.nilsson@math.uu.se

\end{document}